\documentclass[11pt,a4paper]{article}
\usepackage{amsthm,amsmath,amssymb,amsopn,amsfonts,mathrsfs,amsbsy,amscd}
\usepackage{url}
\usepackage[english]{babel}
\usepackage{makeidx}\usepackage{amssymb}

\newcommand{\prs}{\langle\;,\;\rangle}
\newcommand{\too}{\longrightarrow}

\newcommand{\an}{{\pi_{\#}}}

\newcommand{\G}{{\cal G}}

\newcommand{\h}{{\cal H}}

\newcommand{\M}{{\cal M}}

\newcommand{\D}{{\cal D}}

\newcommand{\Om}{\Omega}

\newcommand{\wi}{\widetilde}
\newcommand{\al}{\alpha}
\newcommand{\be}{\beta}
\newcommand{\ga}{\gamma}

\newcommand{\la}{\lambda}

\font\bb=msbm10

\def\C{\hbox{\bb C}}

\newtheorem{theorem}{Theorem}[section]
\newtheorem{proposition}{Proposition}[section]
\newtheorem{lemma}{Lemma}[section]
\newtheorem{corollary}{Corollary}[section]

\newtheorem{remark}{Remark}
\DeclareMathOperator*{\ad}{ad}
\DeclareMathOperator*{\tr}{tr}
\DeclareMathOperator*{\Div}{div}

\title{Multiplicative deformations of spectrale triples associated to left invariant metrics on Lie groups}
\author{Amine BAHAYOU \& Mohamed BOUCETTA}
\date{}

\begin{document}
\maketitle

\begin{abstract} We study the triple $(G,\pi,\prs)$ where $G$ is a connected and simply connected Lie group, $\pi$ and $\prs$ are, respectively,  a multiplicative Poisson tensor and a left invariant Riemannian metric on $G$ such that  the necessary conditions, introduced by Hawkins, to the existence of a non commutative deformation (in the direction of $\pi$) of the spectrale triple associated to $\prs$ are satisfied. We show that the geometric problem of the classification of  such triple $(G,\pi,\prs)$ is equivalent to  an algebraic one.
We solve this algebraic problem in  low dimensions  and we give the list of all
$(G,\pi,\prs)$ satisfying Hawkins's conditions, up to dimension four.
\end{abstract}
\paragraph{keywords:}Poisson-Lie groups, contravariant connections, metacurvature, unimodularity, specrale triples.
\section{Introduction}
In \cite{haw1} and \cite{haw2}, Hawkins showed that if a deformation of the graded
algebra of differential forms on  a Riemannian manifold $(M,\prs)$
 comes from a deformation of the spectral triple
describing the Riemannian manifold $M$, then the Poisson tensor
$\pi$ (which characterizes the deformation) and the Riemannian
metric satisfy the following conditions:
\begin{enumerate}
\item The associated metric contravariant connection $\cal D$ is \emph{flat}.
\item The metacurvature of $\cal D$ vanishes, ($\cal D$ is \emph{metaflat}).
\item The Poisson tensor $\pi$
is compatible with the Riemannian volume $\mu$:
$$d(i_{\pi}\mu)=0.$$
\end{enumerate}
The metric contravariant connection associated naturally to any
couple of pseudo-Riemannian metric and Poisson tensor is an
analogue of the Levi-Civita connection. It has appeared first in
\cite{bou1}. The metacurvature, introduced by Hawkins in \cite{haw2}, is a
$(2,3)$-tensor field (symmetric in the contravariant indices and
antisymmetric in the covariant indices) associated naturally to
any torsion-free and flat contravariant connection.

 In \cite{haw2}, Hawkins studied completely the geometry of the triples $(M,\prs,\pi)$ satisfying 1-3 when $M$ is compact and $\prs$ is Riemannian. In \cite{bou3}, the second author gave a method which permit the construction of a large class of triples $(M,\prs,\pi)$ satisfying 1-3. We call the conditions 1-3 {\it Hawkins's conditions} and a couple $(\pi,\prs)$ satisfying 1-2 will be called {\it flat} and {\it metaflat}.\\

 In this paper, we study the triples $(G,\pi,\prs)$ satisfying Hawkins's conditions, where $G$ is a connected and simply connected Lie group endowed with  a multiplicative Poisson tensor $\pi$  and  a left invariant Riemannian metric $\prs$. We reduce the geometric problem of classifying such triples to  an   algebraic one and we solve it when the dimension of the Lie group is $\leq4$. In \cite{bah}, the authors gave the complete description of the triples
 $(G,\pi,\prs)$ satisfying Hawkins's conditions when $G$ is the $2n+1$-dimensional Heisenberg group.\\

 To state our main results, let us introduce the notion of {\it Milnor Lie algebra} which will be central in this paper and  recall briefly some classical facts about Poisson-Lie groups. The notion of Poisson-Lie group was first introduced by Drinfel'd \cite{dr} and studied by Semenov-Tian-Shansky \cite{sts} (see also \cite{lu-wei}).
 \begin{enumerate}\item A \emph{ Milnor Lie algebra} is a finite dimensional real Lie algebra $\G$ endowed with a scalar product $\prs$ such that:
 \begin{enumerate}\item the Lie subalgebra $S=\{u\in\mathcal{G}, \ad_u+\ad_u^t=0\}$ is abelian ($\ad_u^t$ denotes the adjoint of $\ad_u$ w.r.t. $\prs$),
\item the derived ideal $[\G,\G]$ is abelian   and $S^\perp=[\G,\G]$ ($S^\perp$ is the orthogonal of $S$).\\ This terminology is justified by a classical result of Milnor. Indeed, in  \cite{mi}, Milnor showed that a left invariant Riemannian metric on a Lie group is flat if and only if its Lie algebra is a semi-direct product of an abelian algebra $\mathfrak{b}$ with an abelian ideal $\mathfrak{u}$ and, for any $u\in\mathfrak{b}$, $ad_u$ is skew-symmetric. This result can be formulated in a more precise way and, in Proposition \ref{milnor}, we will show that a left invariant Riemannian metric on a Lie group is flat if and only if its Lie algebra is a Milnor Lie algebra.
\end{enumerate}
\item Let $G$ be a Lie group and $\G$ its Lie algebra. A  Poisson tensor $\pi$ on $G$ is called  multiplicative if, for any $a,b\in G$, $$\pi(ab)=(L_{a})_*\pi(b)+(R_{b})_*\pi(a),$$ where $(L_{a})_*$ (resp. $(R_{b})_*$) denotes the tangent map of the left  translation of $G$ by $a$ (resp. the right  translation of $G$ by $b$). Pulling $\pi$ back to the identity
element $e$ of $G$ by left translations, we get a map
$\pi_\ell:G\too{\cal G}\wedge{\cal G}$
 defined by $\pi_\ell(g)=(L_{g^{-1}})_*\pi(g)$. Let
$\xi:=d_e\pi_\ell:{\cal G}\too{\cal G}\wedge{\cal G}$ be the intrinsic derivative of
$\pi_\ell$ at $e$.  It is well-known that $(\G,[\;,\;],\xi)$ is a Lie bialgebra, i.e., $\xi$ is a $1$-cocycle relative to the adjoint representation of $\G$ on $\G\wedge\G$, and
the dual map of $\xi$, $[\;,\;]^*:{\cal G}^*\times{\cal G}^*\too{\cal G}^*$,
is a Lie bracket on ${\cal G}^*$. It is well-known also that $(\G^*,[\;,\;]^*,\rho)$ is also a Lie bialgebra, where   $\rho:\G^*\too\G^*\wedge\G^*$ is the dual of the Lie bracket on $\G$. Note that $\rho=-d$ where $d$ is the restriction of the differential to left invariant $1$-forms.\\ A Poisson-Lie group endowed with a left invariant Riemannian metric will be called {\it Riemannian Poisson-Lie group}.

For any scalar product $\prs$ on a Lie algebra $\G$, we denote by $\prs^*$ the associated scalar product on $\G^*$.
\end{enumerate}

Let us state our main results:
\begin{theorem}\label{main1}
Let $(G,\pi,\langle\;,\;\rangle)$ be a Riemannian Poisson-Lie group  and $(\mathcal{G}^*,[\,,\,]^*,\rho)$ its dual Lie bialgebra. Then $(\pi,\prs)$  is flat and metaflat if and only if:\begin{enumerate}
 \item $(\mathcal{G}^*,[\,,\,]^*,\langle\;,\;\rangle^*_e)$ is a Milnor Lie algebra, \item for any $\alpha,\beta,\gamma\in S=\{\alpha\in\mathcal{G}^*,\ \ad_\alpha+\ad_\alpha^t=0\}$,
\begin{equation}\label{flat}
\ad\nolimits_\alpha\ad\nolimits_\beta\rho(\gamma)=0.
\end{equation}
\end{enumerate}
\end{theorem}
\begin{theorem}\label{main2} Let $(G,\pi)$ be a connected and unimodular  Poisson-Lie group  and let $\mu$ be a left invariant volume form on $G$. Then $d\left(i_{\pi}\mu\right)=0$ if and only if:\begin{enumerate}
 \item $(\mathcal{G}^*,[\,,\,]^*)$ is an unimodular Lie algebra,
 \item for any $u\in\G$,
\begin{equation}\label{unimodular}
\rho(i_{\xi(u)}\mu_e)=0,
\end{equation}where $\xi$ is the $1$-cocycle associated to $\pi$ and $\rho=-d$ is the dual $1$-cocycle extended as a differential to $\wedge^{\dim\G-2}\G^*$.\end{enumerate}
\end{theorem}
We will see ({\it cf.} Proposition \ref{cn}) that for a general connected Poisson-Lie group the condition $d\left(i_{\pi}\mu\right)=0$ implies \eqref{unimodular}.

If $G$ is abelian then $\rho=0$ and one can deduce easily from Theorems \ref{main1}-\ref{main2} the following result.
\begin{corollary}\label{corollary1}Let $(\G,\prs)$ be a  Lie algebra endowed with a scalar product and denote by $\pi_l$ the canonical linear Poisson structure on $\G^*$. Then $(\pi_\ell,\prs^*)$ satisfies Hawkins's conditions if and only if $(\G,\prs)$ is a Milnor Lie algebra.
\end{corollary}
There are some interesting consequences of Theorems \ref{main1}-\ref{main2}:
\begin{enumerate}\item The classification of connected and simply connected Riemannian Poisson-Lie groups which are flat and metaflat is equivalent to the classification of the Lie bialgebra structures on Milnor Lie algebras for which (\ref{flat}) holds.
\item The classification of unimodular connected and simply connected Riemannian Poisson-Lie groups satisfying Hawkins's conditions  is equivalent to the classification of the Lie bialgebra structures on Milnor Lie algebras for which (\ref{flat}) and (\ref{unimodular}) hold.
        \item The Lie bialgebras structures on Milnor Lie algebras of dimension $\leq4$ can be computed  (see Section 4) and hence the Riemannian Poisson-Lie groups of dimension $\leq 4$ satisfying Hawkins's conditions  can be deduced (see Theorems \ref{dim3} and the paragraph devoted to the 4-dimensional case in Section 4).

\end{enumerate}

The paper is organized as follows. In Section $2$, we present a reformulation of a classical result of Milnor  and we recall some standard facts about Levi-Civita contravariant connections and about the metacurvature of flat and torsion-free contravariant connections. In section $3$, we prove Theorems \ref{main1}-\ref{main2} and finally, Section 4 is devoted to the determination of Riemannian Poisson-Lie groups satisfying Hawkins's conditions in dimension 2, 3 and 4.
\section{Preliminaries}

\subsection{Milnor Lie algebras}
The following lemma is interesting in itself:
\begin{lemma}\label{killing2}
Let $(G,\prs)$ be a Lie group with a left invariant Riemannian metric. If the sectional curvature of $\prs$ is nonpositive then the Lie subalgebra $S=\{u\in\G,\; \ad_u+\ad_u^t=0\}$ is abelian.
\end{lemma}
\begin{proof}
For any $u\in\G$, we denote by $u^+$ the left invariant vector field associated to $u$. Remark that $S^+=\{u^+,\, u\in S\}$ is the Lie algebra of left invariant Killing vector fields. Now, since for any $u\in S$, $\langle u^+,u^+\rangle$ is constant then, for any left invariant vector field $X$ we have:
\begin{equation}\label{killing1}
\langle\nabla_X\nabla_X u^+,u^+\rangle+\langle\nabla_X u^+,\nabla_X u^+\rangle=0,
\end{equation}
where $\nabla$ is the Levi-Civita connection associated to $\prs$.\\
The vector field $u^+$ is Killing, thus we have the well-known formula (see \cite{besse}, Theorem $1.81$)
\begin{equation*}
\nabla_X\nabla_X u^+-\nabla_{\nabla_XX}u^+=R(u^+,X)X
\end{equation*}
where $R(X,Y)=\nabla_{[X,Y]}-[\nabla_X,\nabla_Y]$ is the tensor curvature. Moreover $\langle\nabla_{\nabla_XX}u^+,u^+\rangle=0$, hence the formula \eqref{killing1} becomes:
\begin{equation*}
\langle R(u^+,X)X,u^+\rangle+\langle\nabla_X u^+,\nabla_X u^+\rangle=0.
\end{equation*}
This implies, since the curvature is nonpositive, that $\langle\nabla_X u^+,\nabla_X u^+\rangle=0$. So $u^+$ is a parallel vector field and the lemma follows.
\end{proof}

The following proposition is a reformulation of a classical result of Milnor (see \cite{mi} Theorem 1.5).
\begin{proposition}\label{milnor} Let $(G,\prs)$ be a Lie group endowed with a left invariant Riemannian metric. Then the curvature of $\prs$ vanishes if and only if the Lie algebra $\G$ of $G$ endowed with the scalar product $\prs_e$ is a Milnor Lie algebra.\end{proposition}
\begin{proof} Note first that the Levi-Civita connection of $\prs$ is entirely determined by the product $A:\G\times\G\too\G$ given by
\begin{equation}\label{levi}
2\langle A_uv,w\rangle_e=
\langle[u,v],w\rangle_e+\langle[w,u],v\rangle_e+\langle[w,v],u\rangle_e,\end{equation}
and the curvature vanishes if and only if, for any $u,v\in\G$, $A_{[u,v]}=[A_u,A_v]$.

If $\G=S\oplus[\G,\G]$ is a Milnor Lie algebra, then one can deduce easily from \eqref{levi} that
\begin{equation*}
A_u=\left\{\begin{array}{ll}
0 & \text{if}\ u\in[\G,\G] \\
\ad\nolimits_u & \text{if}\ u\in S,
\end{array}\right.
\end{equation*}and hence the curvature vanishes identically.

Suppose now that the curvature vanishes. In the proof of his result, Milnor considered $\mathfrak{u}=\{u\in\mathcal{G},\ A_u=0\}$ and showed that $\mathfrak{u}$ is an abelian ideal, its orthogonal
$\mathfrak{b}$ is an abelian subalgebra and for all $u\in\mathfrak{b}$, $\ad_u$ is skew-symmetric. Hence $\mathfrak{b}\subset S$ and $[\G,\G]=[\mathfrak{b},\mathfrak{u}]$.

Now, for any $u\in \mathfrak{u}$, $v\in \mathfrak{b}$ and  $w\in\G$, we have  $A_u=0$ and then
$$\langle w,[u,v]\rangle+\langle ad_wu,v\rangle+\langle u,ad_wv\rangle=0.$$This relation implies that $S=[\G,\G]^\perp$. We deduce that  $[\G,\G]\subset\mathfrak{u}$ and $[\G,\G]$ is abelian. From Lemma \ref{killing2}, $S$ is abelian which completes the proof. \end{proof}

\begin{proposition} Let $\G$ be a Milnor Lie algebra. If $\dim S\geq1$ then  the derived ideal $[\G,\G]$ is of even dimension.\end{proposition}
\begin{proof} Let $(s_1,...,s_p)$ be a basis of $S$. The restriction of $\ad_{s_1}$ to $[\G,\G]$ is a skew-symmetric endomorphism, thus its kernel $K_1$ is of even codimension in $[\G,\G]$. Now, $\ad_{s_2}$ commutes with $\ad_{s_1}$ and let invariant $K_1$,
 then $K_1\cap\ker\ad_{s_2}$ is of even codimension in $K_1$ and hence of even codimension in $[\G,\G]$. Thus, by induction, we show that $$K_p=[\G,\G]\cap\left(\cap_{i=1}^p\ker\ad\nolimits_{s_i}\right)$$ is an even codimensional subspace of $[\G,\G]$. Now from its definition  $K_p$ is contained in the center of $\G$ which is contained in $S$ and then $K_p=\{0\}$ and the result follows.
\end{proof}
\subsection{Contravariant connections and metacurvature}
Contravariant connections associated to a Poisson structure have
recently turned out to be useful in several areas of Poisson
geometry. Contravariant connections were defined by Vaisman \cite{vai}
and were analyzed in detail by Fernandes \cite{fer2}. This notion appears
extensively  in the context of noncommutative deformations (see
\cite{haw1,haw2}).

Let $(P,\pi)$ be a Poisson manifold. We consider
$\an:T^*P\longrightarrow TP$  the anchor map given by
$\beta(\an(\alpha))=\pi(\alpha,\beta),$ and we denote by $[\;,\;]_\pi$
the Koszul bracket on differential $1$-forms  given by
\begin{equation}\label{koszul.bracket}
[\alpha,\beta]_\pi=
 \mathscr{L}_{\pi_{\#}(\alpha)}\beta-\mathscr{L}_{\pi_{\#}(\beta)}\alpha-d(\pi(\alpha,\beta)).
\end{equation}
This bracket can be extended naturally to $\Om^*(P)$ and gives rise to a bracket which we denote also by $[\;,\;]_\pi.$

A contravariant connection on $P$, with respect to $\pi$, is a $\mathbb{R}$-bilinear map
\begin{equation*}
\begin{array}{cccc}
\mathcal{D} : &\Omega^1(P)\times\Omega^1(P)&\longrightarrow&\Omega^1(P) \\
   &(\alpha,\beta)&\longmapsto&\mathcal{D}_\alpha\beta
\end{array}
\end{equation*}
satisfying the following properties:
\begin{enumerate}
\item $\alpha\mapsto\mathcal{D}_\alpha\beta$ is $\mathscr{C}^\infty(P)$-linear, that is:
$$\mathcal{D}_{f\alpha}\beta=f\mathcal{D}_\alpha\beta,\ \text{for all}\, f\in\mathscr{C}^\infty(P).$$
\item $\beta\mapsto\mathcal{D}_\alpha\beta$ is a derivation, in the sense:
$$\mathcal{D}_\alpha\left(f\beta\right)=f\mathcal{D}_\alpha\beta+\an(\alpha)(f)\beta,\ \text{for all}\, f\in\mathscr{C}^\infty(P).$$
\end{enumerate}
The torsion and the curvature of a contravariant connection ${\cal D}$ is formally
identical to the usual definitions
$$T(\alpha,\beta)={\cal D}_\alpha\beta-{\cal D}_\beta\alpha-[\alpha,\beta]_\pi\quad\mbox{and}\quad
K(\alpha,\beta)={\cal D}_\alpha{\cal D}_\beta-{\cal D}_\beta{\cal D}_\alpha-{\cal D}_{[\alpha,\beta]_\pi}.$$
The connection ${\cal D}$ is called \emph{flat} if $K$ vanishes identically.\\
Let us define now an interesting  class of contravariant connections, namely Levi-Civita contravariant connections.\\
Let $(P,\pi)$ be a Poisson manifold and $\langle\;,\;\rangle$ a pseudo-Riemannian scalar product
 on $T^*P$.
The metric contravariant connection associated to $(\pi,\langle\;,\;\rangle)$
is the unique contravariant connection ${\cal D}$ such that ${\cal D}$ is torsion-free and the  metric $\langle\;,\;\rangle$ is parallel with respect
to ${\cal D}$, i.e.,
$$\pi_{\#}(\alpha).\langle\beta,\gamma\rangle=\langle{\cal D}_\alpha\beta,\gamma\rangle+\langle\beta,{\cal D}_\alpha\gamma\rangle.$$
 The connection ${\cal D}$ is the contravariant analogue of the Levi-Civita connection and can be defined by the Koszul formula:
\begin{eqnarray}
2\langle{\cal D}_\alpha\beta,\gamma\rangle&=&\pi_{\#}(\alpha).\langle\beta,\gamma\rangle+
\pi_{\#}(\beta).\langle\alpha,\gamma\rangle-
\pi_{\#}(\gamma).\langle\alpha,\beta\rangle\nonumber\\
&+&\langle[\gamma,\alpha]_\pi,\beta\rangle+\langle[\gamma,\beta]_\pi,\alpha\rangle+
\langle[\alpha,\beta]_\pi,\gamma\rangle.
\label{koszul.connection}\end{eqnarray}
We call $\D$ the \emph{Levi-Civita contravariant connection} associated to $(\pi,\prs)$.
\paragraph{The metacurvature}
We recall now the definition of the metacurvature
introduced by Hawkins in \cite{haw2}.

Let $(P,\pi)$ be a Poisson manifold and $\D$ a torsion-free and
flat contravariant connection with respect to $\pi$. In \cite{haw2},
Hawkins showed that such a connection defines  a bracket
$\{\;,\;\}$ on the space of differential forms $\Om^*(P)$ such
that:
\begin{enumerate} \item $\{\;,\;\}$ is $\mathbb{R}$-bilinear, degree 0
and antisymmetric, i.e.,
$$\{\sigma,\rho\}=-(-1)^{deg\sigma deg\rho}\{\rho,\sigma\}.
$$
\item The differential $d$ is a derivation with respect to
$\{\;,\;\}$, i.e.,
$$d\{\sigma,\rho\}=\{d\sigma,\rho\}+(-1)^{deg\sigma}\{\sigma,d\rho\}.
$$ \item $\{\;,\;\}$ satisfies the product rule
$$\{\sigma,\rho\wedge\la\}=\{\sigma,\rho\}\wedge\la+(-1)^{deg\sigma
deg\rho}\rho\wedge\{\sigma,\la\}.$$ \item For any $f,g\in\mathscr{C}^\infty(P)$ and for any $\sigma\in\Om^*(P)$ the bracket $\{f,g\}$
coincides with the initial Poisson bracket and
$$\{f,\sigma\}=\D_{df}\sigma.$$\end{enumerate}
Hawkins called this bracket a \emph{generalized Poisson bracket} and
showed that there exists a $(2,3)$-tensor $\M$ (symmetric in the contravariant indices and
antisymmetric in the covariant indices) such that the
following assertions are equivalent:\begin{enumerate}\item The
generalized Poisson bracket satisfies the graded Jacobi identity
$$\{\{\sigma,\rho\},\la\}=\{\sigma,\{\rho,\la\}\}
-(-1)^{deg\sigma deg\rho}\{\rho,\{\sigma,\la\}\}.
$$\item The tensor $\M$ vanishes identically.\end{enumerate}
$\M$ is called the \emph{metacurvature} and is given by
\begin{equation}\label{metacourbure}\M(df,\al,\be)=\{f,\{\al,\be\}\}-\{\{f,\al\},\be\}-
\{\{f,\be\},\al\}.\end{equation}
The connection ${\cal D}$ is called \emph{metaflat} if ${\cal M}$ vanishes identically.\\

The following formulas, due to Hawkins, will be useful later. Indeed,
Hawkins pointed out in \cite{haw2} pp. 394, that for any parallel 1-form
$\al$ with respect to $\D$ and any 1-form $\be$, the generalized Poisson bracket of
$\al$ and $\be$ is given by
\begin{equation}\label{eq7}\{\alpha,\beta\}=-{\cal D}_{\beta}d\alpha.\end{equation}Thus, one can deduce from
(\ref{metacourbure}) that for any parallel 1-forms $\alpha,\gamma$ and for any 1-form
$\beta$,
\begin{equation}\label{metacourbure2}{\cal M}(\alpha,\beta,\gamma)=-{\cal D}_\beta{\cal D}_\gamma d\alpha.\end{equation}

To finish this section, we give a useful full global formula
for Hawkin's generalized Poisson bracket of two 1-forms. Let $\al$
and $\be$ be two 1-forms on a Poisson  manifold $P$ endowed with a
 torsion-free and flat contravariant connection $\D$. One can suppose
  that $\be=gdf$ where $f,g\in\mathscr{C}^\infty(P)$. Then, we
 have
 \begin{align*}
\{\alpha,fdg\}=&\{\alpha,f\}\wedge dg+f\{\alpha,dg\}\\
=&-\D_{df}\alpha\wedge dg+f\left(d\D_{dg}\alpha-\D_{dg}d\alpha\right)\\
=&-\D_{fdg}d\alpha+d\D_{fdg}\alpha-\D_{df}\alpha\wedge dg-df\wedge\D_{dg}\alpha\\
=&-\D_{fdg}d\alpha+d\D_{fdg}\alpha-\D_\alpha\left(df\wedge dg\right)-[df,\alpha]_\pi\wedge dg-df\wedge[dg,\alpha]_\pi\\
=&-\D_{fdg}d\alpha-\D_\alpha\left(d(fdg)\right)+d\D_{fdg}\alpha-[df,\alpha]_\pi\wedge dg-df\wedge[dg,\alpha]_\pi\\
=&-\D_{fdg}d\alpha-\D_\alpha\left(d(fdg)\right)+d\D_{fdg}\alpha+[\alpha,d(fdg)]_\pi\\
=&-\D_\alpha d\beta-\D_\beta d\alpha+d\D_\beta\alpha+[\alpha,d\beta]_\pi.
\end{align*}
Thus, for any $\al,\be\in\Om^1(P)$, we have
\begin{equation}\{\alpha,\beta\}=-{\cal D}_\alpha d\beta-{\cal D}_\beta d\alpha+d{\cal D}_\beta\alpha+[\alpha,d\beta]_\pi.\label{eq6}\end{equation}

\section{Proofs of Theorems \ref{main1}-\ref{main2}}
\subsection{Proof of Theorem \ref{main1}}
Theorem \ref{main1} is an immediate consequence of the following result.
\begin{theorem}Let $(G,\pi,\langle\;,\;\rangle)$ be a Riemannian Poisson-Lie group. Then:
\begin{enumerate}\item $(\pi,\prs)$ is flat if and only if the dual Lie algebra $(\G^*,\prs^*)$ is a Milnor Lie algebra.
\item If $(\pi,\prs)$ is flat then, if one identify $\G^*$ with the space of left invariant 1-forms, the metacurvature  $\M$ is given by
    \begin{equation}\label{metacourbure3}
\mathcal{M}\left(\alpha,\beta,\gamma\right)=\left\{\begin{array}{ll}
\ad\nolimits_\alpha\ad\nolimits_\beta\rho(\gamma) & \text{for all}\ \alpha,\beta,\gamma\in S, \\
\qquad 0 & \text{otherwise},
\end{array}\right.
\end{equation}where $S=\{\alpha\in\mathcal{G}^*,\ \ad_\alpha+\ad_\alpha^t=0\}$ and $\rho:\G^*\too\G^*\wedge\G^*$ is the dual $1$-cocycle.
\end{enumerate}
\end{theorem}
\begin{proof} Note first that in a Poisson-Lie group the Koszul bracket of two left invariant 1-form is a left invariant 1-form (see \cite{wei}) and, if one identifies $\G^*$ with the space of left invariant 1-forms, the Koszul bracket coincides with the Lie bracket of $\G^*$. Through this proof, we identify $\G^*$ with the space of left invariant 1-forms on $G$.

\begin{enumerate}\item  Denote by $\prs^*$ the left invariant metric on $T^*G$ associated to $\prs$ and denote by $\D$ the Levi-Civita contravariant connection associated to $(\pi,\prs^*)$. Since the Riemannian metric is left invariant, for any $\al,\be,\ga\in\G^*$, \eqref{koszul.connection} becomes
\begin{eqnarray}
2\langle{\cal D}_\alpha\beta,\gamma\rangle^*=
\langle[\gamma,\alpha]_\pi,\beta\rangle^*+\langle[\gamma,\beta]_\pi,\alpha\rangle^*+
\langle[\alpha,\beta]_\pi,\gamma\rangle^*.
\label{koszul.connection1}\end{eqnarray}Hence the restriction of $\D$ to $\G^*\times\G^*$ defines a product on $\G^*$. The vanishing of the curvature of $\D$ is equivalent to the vanishing of the restriction of the curvature of $\D$ to $\G^*$. Now, one can deduce from \eqref{koszul.connection1} that the vanishing of the restriction of the curvature of $\D$ to $\G^*$ is equivalent to the flatness of the left invariant Riemannian metric associated to $\prs^*_e$ on any Lie group with $\G^*$ as a Lie algebra and one can conclude by using Proposition \ref{milnor}.

\item Suppose now that $(\pi,\prs)$ is flat and, according to the first part,  let $\G^*=S\stackrel{\perp}\oplus [\G^*,\G^*]$ where $S=\{\al\in\G^*,ad_\al+ad_\al^t=0\}$ and both $S$ and $[\G^*,\G^*]$ are abelian. Let us establish \eqref{metacourbure3}.

    First, one can deduce from \eqref{koszul.connection1} that, for any $\ga\in\G^*$,
\begin{equation}\label{levi-civita}
\D_\al\ga=\left\{\begin{array}{ll}
0 & \text{if}\ \al\in [\G^*,\G^*] \\
\;[\al,\ga]_\pi=\ad\nolimits_\al\ga & \text{if}\ \al\in S,
\end{array}\right.
\end{equation}and moreover, for any $\al\in S$, $\D\al=0.$

\begin{enumerate}\item If $\al,\be,\ga\in S$, since $\D\al=\D\be=\D\ga=0$, we deduce from \eqref{metacourbure2} that
$$\M(\al,\be,\ga)=-\D_\al\D_\be d\ga\stackrel{\eqref{levi-civita}}=ad_\al ad_\be\rho(\ga).$$
\item If $\al,\ga\in S$ and $\be\in [\G^*,\G^*]$, since $\D\al=\D\ga=0$, we deduce from \eqref{metacourbure2} that
$$\M(\al,\be,\ga)=-\D_\al\D_\be d\ga\stackrel{\eqref{levi-civita}}=0.$$
\item If $\alpha,\beta\in [\G^*,\G^*]$ and $\gamma\in S$. At least locally, we have $\alpha=\sum f_idg_i$ and we deduce from \eqref{metacourbure} that
\begin{equation*}
\mathcal{M}(\alpha,\beta,\gamma)=\sum f_i\{g_i,\{\beta,\gamma\}\}-f_i\{\{g_i,\beta\},\gamma\}-f_i\{\{g_i,\gamma\},\beta\}.
\end{equation*}
From \eqref{eq7}, we have $\{\beta,\gamma\}=-\mathcal{D}_\gamma d\beta=0$, and from  \eqref{levi-civita}, $\{g_i,\gamma\}=\D_{dg_i}\ga=0$, thus
\begin{equation*}
\mathcal{M}(\alpha,\beta,\gamma)=\sum-f_i\{\{g_i,\beta\},\gamma\}=-\sum f_i\mathcal{D}_{\mathcal{D}_{dg_i}\beta}\gamma=
-\mathcal{D}_{\mathcal{D}_\alpha\beta}\gamma=0.
\end{equation*}

\item For $\alpha,\beta\in [\G^*,\G^*]$, the computation of $\M(\al,\be,\be)$ is more difficult. First, by comparing $\M(\al,\be,\be)$ and $[\be,[\be,d\al]_\pi]_\pi$, we will show that they agree up to sign and, next, we will show that $[\be,[\be,d\al]_\pi]_\pi=0$ and we get the result.

Put $\al=\sum f_idg_i$. By using \eqref{metacourbure}, we get
\begin{align*}
\mathcal{M}(\alpha,\beta,\beta)=&\sum f_i\{g_i,\{\alpha,\beta\}\}-2f_i\{\{g_i,\beta\},\beta\}\\
=&\sum f_i\D_{dg_i}\{\alpha,\beta\}-2\sum f_i\{\mathcal{D}_{dg_i}\beta,\beta\}\\
=&\D_\al\{\al,\be\}-2\sum f_i\{\mathcal{D}_{dg_i}\beta,\beta\}\\
\stackrel{(*)}=&-2\sum f_i\{\mathcal{D}_{dg_i}\beta,\beta\}\\
=&-2\sum \left(\{f_i\mathcal{D}_{dg_i}\beta,\beta\}+\mathcal{D}_{df_i}\beta\wedge\mathcal{D}_{dg_i}\beta\right)\\
=&-2\{\mathcal{D}_{\al}\beta,\beta\}-2\sum\mathcal{D}_{df_i}\beta\wedge\mathcal{D}_{dg_i}\beta\\
=&-2\sum\mathcal{D}_{df_i}\beta\wedge\mathcal{D}_{dg_i}\beta.
\end{align*}In $(*)$ we have used \eqref{levi-civita} and the fact that $\{\al,\be\}\in\wedge^2\G^*$ which can be deduced from \eqref{eq6}.
On the other hand,
\begin{align*}
\;[\beta,[\beta,d\alpha]_\pi]_\pi=&\sum[\beta,[\beta,df_i\wedge dg_i]_\pi]_\pi\\
=&\sum[\beta,[\beta,df_i]_\pi\wedge dg_i]_\pi+[\beta,df_i\wedge[\beta,dg_i]_\pi]_\pi\\
=&\sum[\beta,[\beta,df_i]_\pi]_\pi\wedge dg_i+[\beta,df_i]_\pi\wedge[\beta,dg_i]_\pi\\
&+[\beta,df_i]_\pi\wedge[\beta,dg_i]_\pi+df_i\wedge[\beta,[\beta,dg_i]_\pi]_\pi.
\end{align*}
 Now, choose an orthonormal basis $\{\alpha_1,...,\alpha_n\}$ of $\G^*$. For any $\ga\in\Om^1(G)$, we have
$\ga=\sum\langle \ga,\alpha_i\rangle^*\alpha_i$, and
\begin{align*}
[\beta,\ga]_\pi=&\sum\pi_\sharp(\beta)\cdot\langle \ga,\alpha_i\rangle^*\,\alpha_i+
\langle \ga,\alpha_i\rangle^*[\beta,\alpha_i]_\pi\\
=&\mathcal{D}_\beta \ga+\sum\langle \ga,\alpha_i\rangle^*[\beta,\alpha_i]_\pi.
\end{align*}
Hence
\begin{align*}
\;[\beta,[\beta,\ga]_\pi]_\pi=&[\beta,\mathcal{D}_\beta \ga]_\pi+\sum\pi_\sharp(\beta)\cdot\langle \ga,
\alpha_i\rangle^*[\beta,\alpha_i]_\pi+\sum\langle \ga,
\alpha_i\rangle^*[\be,[\beta,\alpha_i]_\pi]_\pi\\
=&[\beta,\mathcal{D}_\beta \ga]_\pi+\sum\langle\mathcal{D}_\be \ga,
\alpha_i\rangle^*[\beta,\alpha_i]_\pi\\
=&2[\beta,\mathcal{D}_\beta \ga]_\pi-\mathcal{D}_\beta\mathcal{D}_\beta \ga\\
=&\mathcal{D}_\beta\mathcal{D}_\beta \ga-2\mathcal{D}_{(\mathcal{D}_\beta \ga)}\beta\\
=&\mathcal{D}_\beta\mathcal{D}_\beta \ga-2\mathcal{D}_{[\beta,\ga]_\pi}\beta-\sum\langle \ga,\alpha_i\rangle^*\D_{[\beta,\alpha_i]_\pi}\be\\
=&\mathcal{D}_\beta\mathcal{D}_\beta \ga-2\mathcal{D}_{[\beta,\ga]_\pi}\beta\\
=&\mathcal{D}_\beta\mathcal{D}_\beta df-2(K(\be,\ga)\be+\mathcal{D}_\beta\mathcal{D}_{\ga}\beta-\D_{\ga}\D_\be\be)\\
=&\mathcal{D}_\beta\mathcal{D}_\beta \ga-2\mathcal{D}_\beta\mathcal{D}_{\ga}\beta.
\end{align*}
By using this formula, we get
\begin{align*}
[\beta,[\beta,df_i]_\pi]_\pi\wedge dg_i=&\mathcal{D}_\beta\mathcal{D}_\beta df_i\wedge dg_i-2\mathcal{D}_\beta\mathcal{D}_{df_i}\beta\wedge dg_i\\
=&\mathcal{D}_\beta\left(\mathcal{D}_\beta df_i\wedge dg_i\right)-\mathcal{D}_\beta df_i\wedge\mathcal{D}_\beta dg_i\\ &-2\mathcal{D}_\beta\left(D_{df_i}\beta\wedge dg_i\right)+2D_{df_i}\beta\wedge\mathcal{D}_\beta dg_i,\\
df_i\wedge[\beta,[\beta,dg_i]_\pi]_\pi=&-\mathcal{D}_\beta\left(\mathcal{D}_\beta dg_i\wedge df_i\right)+\mathcal{D}_\beta dg_i\wedge\mathcal{D}_\beta df_i\\ &+2\mathcal{D}_\beta\left(D_{dg_i}\beta\wedge df_i\right)-2D_{dg_i}\beta\wedge\mathcal{D}_\beta df_i.\\
\end{align*} On the other hand
\begin{align*}
2[\beta,df_i]\wedge[\beta,dg_i]=&2\mathcal{D}_\beta df_i\wedge\mathcal{D}_\beta dg_i-2\mathcal{D}_\beta df_i\wedge\mathcal{D}_{dg_i}\beta\\&-2\mathcal{D}_{df_i}\beta\wedge\mathcal{D}_\beta dg_i+2\mathcal{D}_{df_i}\beta\wedge\mathcal{D}_{dg_i}\beta.
\end{align*}
Thus
\begin{align*}
\;[\beta,[\beta,d\alpha]_\pi]_\pi=&\mathcal{D}_\beta\mathcal{D}_\beta d\alpha+2\sum\mathcal{D}_{df_i}\beta\wedge\mathcal{D}_{dg_i}\beta-
2\mathcal{D}_\beta\left(\mathcal{D}_{df_i}\beta\wedge dg_i\right)\\&-2\mathcal{D}_\beta\left(df_i\wedge\mathcal{D}_{dg_i}\beta\right)\\
=&\mathcal{D}_\beta\mathcal{D}_\beta d\alpha+2\sum\mathcal{D}_{df_i}\beta\wedge\mathcal{D}_{dg_i}\beta+
2\mathcal{D}_\beta\left([\beta,df_i]_\pi\wedge dg_i\right)\\&-2\mathcal{D}_\beta\left(\mathcal{D}_\beta df_i\wedge dg_i\right)
+2\mathcal{D}_\beta\left(df_i\wedge[\beta,dg_i]_\pi\right)-
2\mathcal{D}_\beta\left(df_i\wedge\mathcal{D}_\beta dg_i\right)\\
=&-\mathcal{D}_\beta\mathcal{D}_\beta d\alpha-\mathcal{M}(\alpha,\beta,\beta)+2\mathcal{D}_\beta[\beta,d\alpha]_\pi\\
=&-\mathcal{M}(\alpha,\beta,\beta).
\end{align*}
Now, since $[\G^*,\G^*]$ is abelian and $\beta\in[\G^*,\G^*]$, then $[\beta,[\beta,d\alpha]_\pi]_\pi=0$. This completes the proof. \end{enumerate}

\end{enumerate}\end{proof}

Before  giving a proof for Theorem \ref{main2}, let us show first that, in the general case, the condition \eqref{unimodular} is a necessary condition.
\begin{proposition}\label{cn} Let $(G,\pi)$ be a  Poisson-Lie group and let $\mu$ be a left invariant volume form on $G$. If $d(i_{\pi}\mu)=0$ then \eqref{unimodular} holds.
\end{proposition}

\begin{proof}
The proof is based on the Koszul formula \cite{ko}, satisfied by any vector field $X$ and any multivector $Q$, and  given by
\begin{equation}\label{koszul.formula}
i_{[X,Q]}\mu=i_Xd\left(i_Q\mu\right)+(-1)^{\deg Q}d\left(i_Xi_Q\mu\right)-i_Qd\left(i_X\mu\right).
\end{equation}
Indeed, if $d(i_{\pi}\mu)=0$ then, for any left invariant vector field $X$, we get
$$i_{[X,\pi]}\mu=d\left(i_Xi_\pi\mu\right)-i_\pi d\left(i_X\mu\right).$$
Or $d\left(i_X\mu\right)=\mathscr{L}_X\mu=\al\mu$, where $\al$ is a constant and hence $di_{[X,\pi]}\mu=0.$ One can conclude by using the fact that $[X,\pi]$ is left invariant and $[X,\pi](e)=\xi(X_e)$.
\end{proof}
\subsection{Proof of Theorem \ref{main2}}
\begin{proof} Let $(G,\pi)$ be a connected  unimodular  Poisson-Lie group  and let $\mu$ be a left invariant volume form on $G$.  Let $\xi$ be the $1$-cocycle associated to $\pi$ and let $(G^*,[\,,\,]^*,\rho)$ be the dual Lie bialgebra.
For any tensor $T$ on $\G$, we denote by $T^+$ the corresponding left invariant tensor field on $G$. Recall that the divergence of a vector field $X$ with respect to $\mu$ is the function $\Div_\mu X$ given by
$$\mathscr{L}_{X}\mu=(\Div\nolimits_\mu X)\mu.$$Before giving the proof of the theorem, we need to state some properties of the modular vector field on a  Poisson Lie-group.

As shown in \cite{wei1}, the operator $X_\mu:f\mapsto \Div_\mu{X_f}$  ($X_f$ being the Hamiltonian  vector field associated to $f$) is a derivation and hence a vector field called the \emph{modular vector field} of $(G,\pi)$ with
respect to the volume form $\mu$. It is well-known (see \cite{wei1}) that  $X_\mu$ is given by
\begin{equation}\label{lu1}
d\left(i_\pi\mu\right)=i_{X_\mu}\mu.
\end{equation}
 We define the modular form $\kappa : \mathcal{G}^*\to\mathbb{R}$ by
\begin{equation}
\kappa(\alpha)=\tr\ad\nolimits_\alpha,
\end{equation}
where $\ad_\alpha\beta=[\alpha,\beta]^*$. The modular form $\kappa$, which is in $\mathcal{G}^{**}$, defines a vector  in $\mathcal{G}$ denoted also by $\kappa$. We have
\begin{equation}\label{kappa}
X_\mu(e)=\kappa.
\end{equation}
Indeed, choose a scalar product $\prs$ on $\G$, an orthonormal basis $(u_1,...,u_n)$ of $(\mathcal{G},\langle\,,\,\rangle)$ and denote by  $(\alpha_1,...,\alpha_n)$ its dual basis.  We have $$\pi=\sum_{i<j}\pi_{ij}\,u_i^+\wedge u_j^+$$ and the Hamiltonian vector field associated to $f\in\mathscr{C}^\infty(M)$ is given by $$X_f=\sum_{j=1}^n\left(\sum_{i=1}^n\pi_{ij}\langle df,\al_i^+\rangle^*\right)u_j^+.$$ We have
\begin{align*}
X_\mu(f)=&\Div\nolimits_{\mu}\sum_{j=1}^n\left(\sum_{i=1}^n\pi_{ij}\langle df,\al_i\rangle^*\right)u_j^+\\
=&\sum_{j=1}^n\left(\sum_{i=1}^n\pi_{ij}u_i^+(f)\right)\Div\nolimits_{\mu}u_j^++\sum_{j=1}^n
\sum_{i=1}^nu_j^+\left(\pi_{ij}u_i^+(f)\right).
\end{align*}
Now, since for any $i,j=1,\ldots,n$ $\pi_{ij}(e)=0$ and, because $\G$ is unimodular, $\Div\nolimits_{\mu}u_j^+=0$ for $j=1,\ldots,n$, we get
\[X_\mu(e)=\sum_{i=1}^n\left(\sum_{j=1}^n\mathscr{L}_{X_j^+}\pi(\al_i^+,\al_j^+)_e\right)u_i,\]
and
\begin{align*}
<\alpha_i,X_\mu(e)>=&\sum_{j=1}^n\mathscr{L}_{X_j^+}\pi(\al_i^+,\al_j^+)_e=\sum_{j=1}^n<\alpha_i\wedge\alpha_j,\xi(X_j)>\\
=&\sum_{j=1}^n[\alpha_i,\alpha_j]^*(X_j)=\sum_{j=1}^n\left\langle\sum_{k=1}^n[\alpha_i,\alpha_j]^*(X_k)\alpha_k,\alpha_j\right\rangle^*\\
=&\sum_{j=1}^n\langle[\alpha_i,\alpha_j]^*,\alpha_j\rangle^*=\tr\ad\nolimits_{\alpha_i}=
\kappa(\alpha_i),
\end{align*}and \eqref{kappa} is established.

Now, we will show that $X_\mu-\kappa^+$ is a multiplicative vector field. Indeed,
by applying  \eqref{koszul.formula} and $\mathscr{L}_X\mu=0$, we get
\begin{align*}
i_{[X,X_\mu]}\mu=&\phantom{-}i_Xd\left(i_{X_\mu}\mu\right)-d\left(i_Xi_{X_\mu}\mu\right)-i_{X_\mu}d\left(i_X\mu\right)\\
=&-d\left(i_Xi_{X_\mu}\mu\right),\\
i_{[X,\pi]}\mu=&i_Xd\left(i_\pi\mu\right)+d\left(i_Xi_\pi\mu\right)-i_\pi d\left(i_X\mu\right)\\
=&i_Xi_{X_\mu}\mu+d\left(i_Xi_\pi\mu\right).
\end{align*}
Thus
\begin{equation}\label{formula}
d\left(i_{[X,\pi]}\mu\right)=-i_{[X,X_\mu]}\mu.
\end{equation}
Since $[X,\pi]$ and $\mu$ are left invariant, we deduce from \eqref{formula} that $[X,X_\mu]$ is also left invariant. Moreover, $[X,X_\mu-\kappa^+]=[X,X_\mu]-[X,\kappa^+]$ is left invariant and, since $X_\mu(e)=\kappa^+(e)$, we deduce that $X_\mu-\kappa^+$ is a multiplicative vector field. Thus $X_\mu=X_m+\kappa^+$ where $X_m$ is a multiplicative vector field.

To complete the proof, note that $X_\mu=0$ if and only if $\kappa=0$ and $X_m=0$, i.e., $(\mathcal{G}^*,[\,,\,]^*)$ is unimodular, and  $[X,X_m](e)=0$, for all left invariant vector field $X$. Or the last condition is equivalent, according to \eqref{formula}, to $\rho\left(i_{\xi(u)}\mu\right)=0,$ for any $u\in\G$.
\end{proof}
\section{Examples}
This Section is devoted to the determination of Riemannian Poisson-Lie groups satisfying Hawkins's conditions in the linear case, in dimension $2$, $3$ and $4$.
\paragraph{The linear case} Let $\mathcal{G}=S\oplus [\G,\G]$ be a Milnor Lie algebra. Since $S$ is abelian and acts on $[\G,\G]$ by skew-symmetric endomorphisms, there exists a family of non nul vectors   $u_1,\ldots,u_r\in S$  and an orthonormal basis $(f_1,\ldots,f_{2r})$ of $[\G,\G]$ such that, for any $j=1,\ldots,r$ and for all $s\in S$,
\begin{equation}\label{dual.bialgebra}
[s,f_{2j-1}]=\langle s,u_j\rangle f_{2j}\quad\mbox{and}\quad
[s,f_{2j}]=-\langle s,u_j\rangle f_{2j-1}.
\end{equation}
According to Corollary \ref{corollary1}, the  triple $(\G^*,\pi_\ell,\prs^*)$ satisfies Hawkins's conditions. It is easy to show that there exists a family of constants $(a_{ij})_{1\leq i,j\leq q}$ such that $(\G^*,\pi_\ell,\prs^*)$ is isomorphic to $(\mathbb{R}^{q+2r},\pi_0,\prs_0)$ where $\prs_0$ is the canonical Euclidian metric and
$$\pi_0=\sum_{i=1}^r \left(a_{1i}\partial_{x_1}+\ldots+a_{qi}\partial_{x_q}\right)\wedge\left(y_{2i}\partial_{y_{2i-1}}-
y_{2i-1}\partial_{y_{2i}}\right).$$
\paragraph{The 2-dimensional case} According to Theorems \ref{main1}-\ref{main2} and since any $2$-dimensional Milnor Lie algebra is abelian, a $2$-dimensional connected and simply connected Riemannian Poisson-Lie group $(G,\pi,\prs)$ satisfies Hawkins's conditions if and only if the Poisson tensor is trivial.
\paragraph{The 3-dimensional case} In this paragraph we will determine, up to isomorphism, all the $3$-dimensional connected and simply connected Riemannian Poisson-Lie groups satisfying Hawkins's conditions. According to Theorems \ref{main1}-\ref{main2} and Proposition \ref{cn}, the first step is to determine all the  Lie bialgebra structures on $3$-dimensional Milnor Lie algebras satisfying (\ref{flat}) and (\ref{unimodular}).

Let $\h$ be a $3$-dimensional Milnor Lie algebra. By virtue of \eqref{dual.bialgebra}, there exists a real number $\la\not=0$ and an orthonormal basis $(e_1,e_2,e_3)$ of $\h$ such that
$$[e_2,e_3]=0,\quad [e_1,e_2]=\la e_3\quad\mbox{et}\quad [e_1,e_3]=-\la e_2.$$
We are looking for the 1-cocycles $\rho:\h\too\h\wedge\h$ defining a  Lie bialgebra structure on $\h$ and satisfying (\ref{flat}) and (\ref{unimodular}). Put $$\rho(e_1)=ae_1\wedge e_2+be_1\wedge e_3+ce_2\wedge e_3.$$ The condition  condition (\ref{flat}) is equivalent to
$$ad_{e_1}\circ ad_{e_1}\rho(e_1)=0.$$
We have $ad_{e_1}\rho(e_1)=a\la e_1\wedge e_3-b\la e_1\wedge e_2$ and hence
$$ad_{e_1}\circ ad_{e_1}\rho(e_1)=-a\la^2 e_1\wedge e_2-b\la^2 e_1\wedge e_3.$$ Thus   $\rho$ satisfies (\ref{flat}) if and only if
$$\rho(e_1)=ce_2\wedge e_3.$$
Now put
$$\rho(e_2)=a_1e_1\wedge e_2+b_1e_1\wedge e_3+c_1e_2\wedge e_3 ,\quad
\rho(e_3)=a_2e_1\wedge e_2+b_2e_1\wedge e_3+c_2e_2\wedge e_3,$$
and write down the cocycle condition $\rho([u,v])=ad_u\rho(v)-ad_v\rho(u)$. We get
\begin{eqnarray*}\label{co1}\rho([e_2,e_3])&=&-\la a_2e_3\wedge e_2-\la b_1e_2\wedge e_3=\la(a_2-b_1)e_2\wedge e_3=0,\\
\rho([e_1,e_2])&=&\la(a_1 e_1\wedge e_3- b_1e_1\wedge e_2)=\la\rho(e_3),\\
\rho([e_1,e_3])&=&\la(a_2e_1\wedge e_3-b_2e_1\wedge e_2)=-\la\rho(e_2).
\end{eqnarray*}
These relations are equivalent to
$$b_1=a_2=c_1=c_2=0\quad\mbox{and}\quad a_1=b_2.$$Thus $\rho$ is a 1-cocycle satisfying (\ref{flat}) if and only if
\begin{equation}\label{rho3}\rho(e_1)=ce_2\wedge e_3,\quad \rho(e_2)=ae_1\wedge e_2\quad\mbox{and}\quad
\rho(e_3)=ae_1\wedge e_3.\end{equation}

We consider now $\h^*$ endowed with the bracket associated to $\rho$, the dual scalar product and the dual of the bracket on $\h$, $\xi:\h^*\too\h^*\wedge\h^*$,  given by
\begin{equation}\label{cocycle3}
\xi(e_1^*)=0,\quad\xi(e_2^*)=-\la e_1^*\wedge e_3^*\quad\mbox{and}\quad\xi(e_3^*)=\la e_1^*\wedge e_2^*,\end{equation}where
$(e_1^*,e_2^*,e_3^*)$ is the dual basis of $(e_1,e_2,e_3)$. The bracket on $\h^*$ associated to $\rho$ is given by
\begin{equation}\label{bracket3}[e_1^*,e_2^*]=ae_2^*,\quad [e_1^*,e_3^*]=ae_3^*\quad\mbox{and}\quad
[e_2^*,e_3^*]=ce_1^*.\end{equation}Note that
$$\tr ad_{e_1^*}=2a,\quad \tr ad_{e_2^*}=\tr ad_{e_3^*}=0.$$
The Jacobi identity is given by
\begin{eqnarray*}
\;[[e_1^*,e_2^*],e_3^*]+
[[e_2^*,e_3^*],e_1^*]+
[[e_3^*,e_1^*],e_2^*]=2ace_1^*.\end{eqnarray*}

Let us write down (\ref{unimodular}).  Since $\mu=e_1\wedge e_2\wedge e_3$ and by virtue of (\ref{cocycle3}), a straightforward calculation using \eqref{rho3} gives
\begin{align*}
\rho\left(i_{\xi(e_2^*)}\mu\right)=&\la\rho(e_2)=\la ae_1\wedge e_2,\\
\rho\left(i_{\xi(e_3^*)}\mu\right)=&\la\rho(e_3)=\la ae_1\wedge e_3.
\end{align*}
  In conclusion, $\rho$ defines a Lie bialgebra structure on $\h$ and satisfies (\ref{flat}) and (\ref{unimodular}) if and only if
 \begin{equation}\label{rho3f}\rho(e_1)=ce_2\wedge e_3\quad \mbox{and}\quad \rho(e_2)=
\rho(e_3)=0.\end{equation}Note that in this case, the Lie algebra $\h^*$ is unimodular.
 The following Proposition summarize all the discussion above.
 \begin{proposition} Let $(G,\pi,\prs)$ be a 3-dimensional connected and simply connected Riemannian Poisson-Lie group and let $(\G,\xi,\prs_e)$ be its Lie algebra endowed with the cocycle $\xi$ associated to $\pi$ and the value of the Riemannian metric at the identity. Then $(G,\pi,\prs)$ satisfies Hawkins's conditions if and only if the triple $(\G,\xi,\prs_e)$
 is isomorphic to one of the following triples:
 \begin{enumerate}\item $(\mathbb{R}^3,\xi_0,\prs_0)$ where $\mathbb{R}^3$ is endowed with its abelian Lie algebra structure, $\xi_0$ is given by
 $$\xi_0(e_1)=0,\quad \xi(e_2)=-\la e_1\wedge e_3\quad\mbox{and}\quad\xi(e_3)=\la e_1\wedge e_2,\;\;\la\not=0,$$and $\prs_0$ is the canonical Eucldian scalar product on $\mathbb{R}^3$.
 \item $(\h_3,\xi_0,\prs_0)$ where $\h_3$ the Heisenberg Lie  algebra
 $\left\{\left(\begin{array}{ccc}0&x&z\\0&0&y\\0&0&0\end{array}\right),x,y,z,\in\mathbb{R}^3\right\}$,
 $\xi_0$ is given by
 $$\xi_0(e_3)=0,\quad \xi(e_1)=-\la e_3\wedge e_2\quad\mbox{and}\quad\xi(e_2)=\la e_3\wedge e_1,\;\;\la\not=0,$$and $\prs_0$ is the scalar product on $\h_3$ whose matrix in $(e_1,e_2,e_3)$ is given by
 $\left(\begin{array}{ccc}1&0&0\\0&1&0\\0&0&a\end{array}\right)$, $a>0$.
 \end{enumerate}\end{proposition}

 The infinitesimal situations in this Proposition can be integrated easily which leads to the following theorem.

 \begin{theorem}\label{dim3} Let $(G,\pi,\prs)$ be a connected and simply connected 3-dimensional Riemannian Poisson-Lie group. If $(\pi,\prs)$ satisfies Hawkins's conditions then $(G,\pi,\prs)$ is isomorphic to:
\begin{enumerate}\item $(\mathbb{R}^3,\pi,\prs)$ where $\mathbb{R}^3$ is endowed with its abelian Lie group structure, $\prs$ is the canonical Euclidian metric and
$$\pi=\la\partial_x\wedge(z\partial_y-y\partial_z),$$where $\la\in\mathbb{R}^*$ or,
\item $(H_3,\pi,\prs)$ where $H_3=\left\{\left(\begin{array}{ccc}1&x&z\\0&1&y\\0&0&1\end{array}\right),x,y,z,\in\mathbb{R}^3\right\}$ and $$\pi=\la(x\partial_y-y\partial_x)\wedge\partial_z,\;
    \prs=dx^2+dy^2+a(dz-xdy)^2,$$ where
    $\la\in\mathbb{R}^*$ and $a>0$.\end{enumerate}\end{theorem}

\paragraph{The 4-dimensional case}

In this paragraph we will determine, up to isomorphism, all the 4-dimensional Riemannian Poisson-Lie groups satisfying Hawkins's conditions. According to Theorems \ref{main1}-\ref{main2} Proposition \ref{cn}, the first step is to determine all the  Lie bialgebra structures on 4-dimensional Milnor Lie algebras satisfying (\ref{flat}) and (\ref{unimodular}).

Let  $\h$ be a 4-dimensional Milnor Lie algebra. By virtue of \eqref{dual.bialgebra}, there exists non nul real numbers $\la_1,\la_2$ and an orthonormal basis $(s_1,s_2,f_1,f_2)$ of $\h$ such that
$$[s_1,s_2]=[f_1,f_2]=0,\quad [s_i,f_1]=\la_i f_2\quad\mbox{and}\quad[s_i,f_2]=-\la_i f_1.$$
Put $e_1=\frac{\la_2s_1-\la_1s_2}{\|\la_2s_1-\la_2s_2\|}$. Then there exists $e_2\in S$ such that $(e_1,e_2,f_1,f_2)$ is an orthogonal basis,
$$[e_2,f_1]= f_2,\quad [e_2,f_2]=- f_1,$$and all the other brackets vanish. Note that $\|e_1\|=\|f_1\|=\|f_2\|=1$.

We are looking for the 1-cocycles $\rho:\h\too\h\wedge\h$ defining a  Lie bialgebra structure on $\h$ and satisfying (\ref{flat}) and (\ref{unimodular}).
Put $$\rho(e_i)=a_ie_1\wedge e_2+b_ie_1\wedge f_1+c_ie_1\wedge f_2+
d_ie_2\wedge f_1+f_ie_2\wedge f_2+g_if_1\wedge f_2.$$We have
\begin{eqnarray*}
ad_{e_2}\rho(e_i)&=& b_ie_1\wedge f_2- c_ie_1\wedge f_1+
 d_ie_2\wedge f_2- f_ie_2\wedge f_1,\\
ad_{e_2}\circ ad_{e_2}\rho(e_i)&=&-b_ie_1\wedge f_1-c_ie_1\wedge f_2-
d_ie_2\wedge f_1-f_ie_2\wedge f_2.\end{eqnarray*}

Thus $\rho$ satisfies (\ref{flat}) if and only if, for $i=1,2$,
 $$\rho(e_i)=\al_ie_1\wedge e_2+\be_if_1\wedge f_2.$$
Now, put
\begin{eqnarray*}
\rho(f_i)&=&a_ie_1\wedge e_2+b_ie_1\wedge f_1+c_ie_1\wedge f_2+
d_ie_2\wedge f_1+g_ie_2\wedge f_2+h_if_1\wedge f_2,
\end{eqnarray*}and write down the cocycle condition $\rho([u,v])=ad_u\rho(v)-ad_v\rho(u)$. First, we get
\begin{eqnarray*}
\rho([f_1,f_2])&=&- a_2e_1\wedge f_2
- d_2f_2\wedge f_1- a_1e_1\wedge f_1
-g_1f_1\wedge f_2=0,\end{eqnarray*}thus
$$a_1=a_2=0\quad\mbox{and}\quad d_2-g_1=0.$$
On the other hand,
\begin{eqnarray*}
\rho([e_1,f_1])&=&\al_1e_1\wedge f_2=0,\\
\rho([e_2,f_1])&=& b_1e_1\wedge f_2- c_1e_1\wedge f_1+
 d_1e_2\wedge f_2- g_1e_2\wedge f_1+\al_2 e_1\wedge f_2\\&=&\rho(f_2),\\
\rho([e_1,f_2])&=&-\al_1 e_1\wedge f_1=
0,\\
\rho([e_2,f_2])&=& b_2e_1\wedge f_2- c_2e_1\wedge f_1+
 d_2e_2\wedge f_2- g_2e_2\wedge f_1-\al_2e_1\wedge f_1\\&=&
- \rho(f_1).
\end{eqnarray*}
These relations are equivalent to$$b_2=-c_1,\; c_2=b_1,\; d_2=-g_1,\; g_2=d_1=\al_i=h_i=0.$$
Hence, $\rho$ is a 1-cocycle satisfying (\ref{flat}) if and only if
\begin{equation}\label{rho4}\begin{array}{ccl}
\rho(e_i)&=&\be_i f_1\wedge f_2,\\
\rho(f_1)&=&be_1\wedge f_1+ce_1\wedge f_2+
de_2\wedge f_1,\\
\rho(f_2)&=&-ce_1\wedge f_1+be_1\wedge f_2+de_2\wedge f_2.\end{array}\end{equation}We consider now $\h^*$ endowed with the bracket associated to $\rho$, the dual scalar product and the dual of the bracket on $\h$, $\xi:\h^*\too\h^*\wedge\h^*$,  given by
\begin{equation}\label{cocycle4}\begin{array}{ccl}
\xi(e_1^*)&=&\xi(e_2^*)=0,\\
\xi(f_1^*)&=&- e_2^*\wedge f_2^*,\\
\xi(f_2^*)&=& e_2^*\wedge f_1^*,\end{array}\end{equation}where
$(e_1^*,e_2^*,f_1^*,f_2^*)$ is the dual basis of $(e_1,e_2,f_1,f_3)$. The bracket on $\h^*$ associated to $\rho$ is given by
\begin{equation}\label{bracket4}\begin{array}{lll}
\;[e_1^*,e_2^*]=0,&
\;[e_1^*,f_1^*]=bf_1^*-cf_2^*,&
\;[e_1^*,f_2^*]=cf_1^*+bf_2^*,\\
\;[e_2^*,f_1^*]=df_1^*,&
\;[e_2^*,f_2^*]=df_2^*,&\;[f_1^*,f_2^*]=\be_1e_1^*+\be_2e^*_2.\end{array}\end{equation}Note that
\begin{equation}\label{trace4}\tr ad_{e_1^*}=2b,\quad \tr ad_{e_2^*}=2d,\quad \tr ad_{f_1^*}=\tr ad_{f_3^*}=0.\end{equation}
The Jacobi identities are given by:
 \begin{eqnarray*}
 \;[[e_1^*,e_2^*],f_1^*]+
 [[e_2^*,f_1^*],e_1^*]+[[f_1^*,e_1^*],e_2^*]&=&0,\\
\;[[e_1^*,e_2^*],f_2^*]+[[e_2^*,f_2^*],e_1^*]+[[f_2^*,e_1^*],e_2^*]&=&0,\\
 \;[[e_1^*,f_1^*],f_2^*]+[[f_1^*,f_2^*],e_1^*]+[[f_2^*,e_1^*],f_1^*]&=&2b[f_1^*,f_2^*],\\
 \;[[e_2^*,f_1^*],f_2^*]+[[f_1^*,f_2^*],e_2^*]+[[f_2^*,e_2^*],f_1^*]&=&
 2d[f_1^*,f_2^*].\\
 \end{eqnarray*}
Let us write down (\ref{unimodular}).  Since $\mu=e_1\wedge e_2\wedge f_1\wedge f_2$  and by virtue of (\ref{cocycle4}),
a straightforward computation using \eqref{rho4} gives
\begin{eqnarray*}
\rho\left(i_{\xi(f_1^*)}\mu\right)&=& de_1\wedge e_2\wedge f_1,\\
\rho\left(i_{\xi(f_2^*)}\mu\right)&=& de_1\wedge e_2\wedge f_2.
\end{eqnarray*}

The following proposition summarize all the computation above.
\begin{proposition} Let $(G,\pi,\prs)$ be a 4-dimensional connected and simply connected Riemannian Poisson-Lie group and let $(\G,\xi,\prs_e)$ be its Lie algebra endowed with the cocycle $\xi$ associated to $\pi$ and the value of the Riemannian metric at the identity. If $(G,\pi,\prs)$ satisfies Hawkins's conditions then the triple $(\G,\xi,\prs_e)$
 is isomorphic to $(\mathbb{R}^4,\xi_0,\prs_0)$ where:
\begin{enumerate}\item in the canonical basis $(e_0,e_1,e_2,e_3)$ of $\mathbb{R}^4$, the Lie bracket is given by
\begin{equation*}\begin{array}{lll}
\;[e_1,e_2]=be_2-ce_3,&
\;[e_1,e_3]=ce_2+be_3,&
\;[e_2,e_3]=\be_1e_0+\be_2e_1,\\\;[e_0,e_i]=0,\quad i=1,2,3,&&\\\end{array}\end{equation*}and
$$b[e_2,e_3]=0;$$
\item the cocycle $\xi_0$ is given, up to a  multiplicative constant, by
$$\xi_0(e_0)=\xi_0(e_1)=0,\quad \xi_0(e_2)= e_0\wedge e_3,\quad \xi_0(e_3)=- e_0\wedge e_2,$$
\item the product $\prs_0$ is the canonical Euclidian scalar product of $\mathbb{R}^4$.\end{enumerate}\label{proposition4} \end{proposition}

\begin{remark} When $b=0$, the Lie algebra structure of $\mathbb{R}^4$ given in Proposition \ref{proposition4} is unimodular and, according to Theorem \ref{main2}, the converse of Proposition \ref{proposition4} is true, i.e., the triple $(G,\pi,\prs)$ integrating $(\mathbb{R}^4,\xi_0,\prs_0)$ satisfies Hawkins's conditions.

 However, when $b\not=0$, the triple $(G,\pi,\prs)$ integrating $(\mathbb{R}^4,\xi_0,\prs_0)$ is flat and metaflat and one must check if the last Hawkins's condition is satisfied. We will see that it does.

\end{remark}

The task now is the construction of the triples $(G,\pi,\prs)$ associated to the different models isomorphic to the triple $(\mathbb{R}^4,\xi_0,\prs_0)$ given in Proposition \ref{proposition4}. The computation is very long so we omit it. Note that the determination of the Lie groups is easy since all the models of Lie algebras are product or semi-direct product. The determination of the multiplicative Poisson tensor from the 1-cocycle is a direct calculation using the method exposed in \cite{du} Theorem 5.1.3.

\begin{enumerate}\item {\bf Unimodular case} $b=0$.
\begin{enumerate}\item If $c=\be_1=\be_2=0$ then $(G,\pi,\prs)$ is isomorphic to $(\mathbb{R}^4,\pi_0,\prs_0)$ where $\mathbb{R}^4$ is endowed with its abelian Lie group structure and
$$\pi_0=\partial_x\wedge\left(z\partial_t-t\partial_z\right)\quad\mbox{and}\quad
\prs_0=dx^2+dy^2+dz^2+dt^2.$$
\item If $c=0$ and $\be_1\not=0$ then $(G,\pi,\prs)$ is isomorphic to $(H_0,\pi_0,\prs_0)$ where
    $$H_0=\left\{\left(\begin{array}{cccc}x&0&0&0\\0&1&y&t\\0&0&1&z\\0&0&0&1\end{array}\right),
    x>0,y,z,t\in\mathbb{R}\right\},$$
    $$\be_1\pi_0=(\partial_t-\be_2x\partial_x)\wedge(y\partial_z-z\partial_y)
 +\frac12\be_2(z^2-y^2)x\partial_x\wedge\partial_t,$$and
 $$\prs_0=(x^{-1}dx+\be_2dt-\be_2ydz)^2+dy^2+dz^2+\be_1^2(dt-ydz)^2.$$
 \item If $c=0$ $\be_1=0$ and $\be_2\not=0$ then $(G,\pi,\prs)$ is isomorphic to $(H_0,\pi_0,\prs_0)$ where
    $$H_0=\left\{\left(\begin{array}{cccc}x&0&0&0\\0&1&y&t\\0&0&1&z\\0&0&0&1\end{array}\right),
    x>0,y,z,t\in\mathbb{R}\right\},$$
    $$\pi_0=x\partial_x\wedge(y\partial_z-z\partial_y)
 +\frac12(y^2-z^2)x\partial_x\wedge\partial_t,$$and
 $$\prs_0=\frac1{x^2}dx^2+dy^2+dz^2+\be_2^2(dt-ydz)^2.$$

\item If $c\not=0$ and $(\be_1,\be_2)=(0,0)$ then $(G,\pi,\prs)$ is isomorphic to
$(\mathbb{R}^4,\pi_0,\prs_0)$ where $\mathbb{R}^4$ is endowed with the Lie group structure given by
$$u.v=(x+x',y+y',z+z'\cos y+t'\sin y,t-z'\sin y+t'\cos y)$$when
$u=(x,y,z,t)$ and $v=(x',y',z',t')$, and
$$\pi_0=\partial_x\wedge(z\partial_t-t\partial_z)\quad\mbox{and}
\quad \prs_0=dx^2+ady^2+dz^2+dt^2,$$where $a>0$.
\item If $c\not=0$, $\be_2=0$ and $\be_1\not=0$ then $(G,\pi,\prs)$  is isomorphic to $(\mathbb{R}^2\times\C,\pi_0,\prs_0)$ where $\mathbb{R}^2\times\C$ is endowed with the structure of oscillator group given by
$$(t,s,z).(t',s',z')=\left(t+t',s+s'+\frac12\mbox{Im}\left(\bar{z}exp(it)z'\right),z+exp(it)z'\right),$$
and
$$\pi_0=\partial_s\wedge(x\partial_y-y\partial_x),\quad \prs_0=a dt^2+bds^2+ds(ydx-xdy)+\frac14(ydx-xdy)^2,$$where $a>0$ and $b>0$.

\item If $c\not=0$, $\be_2\not=0$ then $(G,\pi,\prs)$ is isomorphic to $(\mathbb{R}\times G_0,\pi_0,\prs_0)$ where $\mathbb{R}\times G_0$ is the direct product of the abelian group $\mathbb{R}$ with $G_0$ where $G_0$ is either $SU(2)$ or $\wi{SL(2,\mathbb{R})}$ and
    if $\{E_1,E_2,E_3\}$ is a the basis of the Lie algebra of $G_0$ satisfying
$$[E_1,E_2]=E_3,\; [E_3,E_1]=E_2\quad\mbox{and}\quad[E_2,E_3]=\pm E_1$$then
$$\pi=\partial_t\wedge(E_1^+-E_1^-)$$where $E^+_1$ (resp. $E_1^+$) is the left invariant  (resp. right invariant) vector field associated to $E_1$. On the other hand, $\prs_0$ is the left invariant Riemannian metric on $\mathbb{R}\times G_0$ whose value at the identity has the following matrix in the basis $\{E_0,E_1,E_2,E_3\}$
$$\left(\begin{array}{cccc}a&b&0&0\\0&c&0&0\\0&0&d&0\\0&0&0&d\end{array}\right).$$

\end{enumerate}

\item {\bf the non unimodular case:} $b\not=0$. In this case  $(G,\pi,\prs)$ is isomorphic to
$(\mathbb{R}^4,\pi_0,\prs_0)$ where $\mathbb{R}^4$ is endowed with the Lie group structure given by
\begin{eqnarray*}uv&=&\left(x+x',y+y',z+e^{xb}(z'\cos(xc)+t'\sin(xc)),
t+e^{xb}(-z'\sin(xc)+t'\cos(xc))\right).\end{eqnarray*}
when
$u=(x,y,z,t)$ and $v=(x',y',z',t')$,
$$\pi_0=\partial_y\wedge(z\partial_t-t\partial_z)\quad\mbox{and}
\quad \prs_0=dx^2+dy^2+e^{-2bx}(dz^2+dt^2).$$

The Riemannian volume is given by
$$\mu=e^{-2bx}dx\wedge dy\wedge dz\wedge dt,$$and
$$i_{\pi}\mu=-e^{-2bx}(zdx\wedge dz+tdx\wedge dt).$$
Thus $d\left(i_{\pi}\mu\right)=0,$ and the third Hawkins's condition is satisfied.
\end{enumerate}
\paragraph{Acknowledgement}
Amine BAHAYOU would like to thank Philippe Monnier for very useful discussions and Emile Picard Laboratory, at Paul Sabatier University of Toulouse (France), for hospitality where a part of this work was done.


Amine BAHAYOU, Universit\'e Kasdi Merbah,\\
B.P 511 Route de Ghardaïa, 30000 Ouargla Algeria\\
\url{amine.bahayou@gmail.com}\\               
Mohamed BOUCETTA,\\
Facult\'e des sciences et techniques Gueliz\\BP 549 Marrakech Maroc\\
\url{mboucetta2@yahoo.fr}

\begin{thebibliography}{99}
\bibitem[1]{bah} A. Bahayou and M. Boucetta, \emph{Multiplicative noncommutative deformations of left invariant Riemannian metrics on Heisenberg groups}, C. R. Acad. Sci. Paris, S\'erie I 347 (2009), 791-796 .
\bibitem[2]{besse} A. L. Besse, \emph{Einstein Manifolds}, Springer-Verlag 2002.
 \bibitem[3]{bou1} M. Boucetta, \emph{Compatibilit\'e des structures pseudo-riemanniennes et des structures de Poisson}, C. R. Acad. Sci. Paris, {\bf t. 333}, S\'erie I, (2001) 763--768.
\bibitem[4]{bou3} M. Boucetta, \emph{Solutions of the classical Yang-Baxter equation and non-commutative deformations}, {Letters in Mathematical Physics} (2008) 83:69-81.
\bibitem[5]{dr}V. G. Drinfel'd, \emph{Hamiltonian structures on Lie groups, Lie bialgebras and the geometric meaning of the classical Yang-Baxter equations},  Sov. Math. Dokl. {\bf 27} (1) (1983), 68-71.

    \bibitem[6]{du} { J. P. Dufour and N. T. Zung,} \emph{Poisson Structures and Their Normal Forms}, vol. {\bf 242} of Progress in Mathematics. Birkhauser Verlag, Basel, Boston, New York, 2005.

\bibitem[7]{fer2} R. L. Fernandes, \emph{Connections in Poisson geometry. I.
  Holonomy and invariants},  J. Differential Geom. {\bf 54} (2000), no. 2, 303--365.
\bibitem[8]{haw1} E. Hawkins, \emph{Noncommutative rigidity}, Commun. Math. Phys. {\bf 246} (2004), 211-235.
\bibitem[9]{haw2} E. Hawkins, \emph{The structure of noncommutative deformations}, J. Diff. Geom. {\bf 77}, 385-424 (2007).

    \bibitem[10]{ko} { J. L. Koszul}, \emph{Crochet de Schouten-Nijenhuis et Cohomologie,} Ast\'erisque (1985), Numéro Hors S\'erie, 257-271.

\bibitem[11]{lu-wei} J. H. Lu, A. Weinstein, \emph{Poisson Lie groups, dressing transformations and Bruhat decompositions}, J. Diff. Geo. {\bf 31} (1990), 501-526.
\bibitem[12]{mi} J. Milnor, \emph{Curvature of left invariant metrics on Lie groups}, Adv. in Math. {\bf 21} (1976), 283-329.
\bibitem[13]{sts} M. A. Semenov-Tian-Shansky, \emph{Dressing transformations and Poisson Lie group actions}, Publ. RIMS, Kyoto University {\bf 21} (1985), 1237-1260.
\bibitem[14]{vai} I. Vaisman, \emph{Lecture on the geometry of Poisson manifolds}, Progr. In Math. {\bf Vol. 118}, Birkhausser, Berlin, (1994).

    \bibitem[15]{wei} A. Weinstein,\emph{Some remarks on dressing transformations}, J.
Fac. Sci. Univ. Tokyo. Sect. 1A, Math. {\bf36} (1988) 163-167.

\bibitem[16]{wei1} A. Weinstein, {The Modular Automorphism Group of a Poisson
Manifold},  J. Geom. Phys. {\bf23}, (1997) 379-394.
\end{thebibliography}
\end{document}